\documentclass[10pt]{article}
\usepackage{amssymb,latexsym,amsmath,epsfig,amsthm} 
\usepackage{mathtools}
\usepackage{array}

\makeatletter

\renewcommand\section{\@startsection {section}{1}{\z@}
{-30pt \@plus -1ex \@minus -.2ex}
{2.3ex \@plus.2ex}
{\normalfont\normalsize\bfseries}}

\renewcommand\subsection{\@startsection{subsection}{2}{\z@}
{-3.25ex\@plus -1ex \@minus -.2ex}
{1.5ex \@plus .2ex}
{\normalfont\normalsize\bfseries}}

\renewcommand{\@seccntformat}[1]{\csname the#1\endcsname. }

\makeatother

\numberwithin{equation}{section}

\newtheorem{thm}{Theorem}
\newtheorem{lem}{Lemma}
\newtheorem{col}{Corollary}
\numberwithin{thm}{section}
\numberwithin{prop}{section}
\numberwithin{conj}{section}
\numberwithin{col}{thm}
\numberwithin{lem}{section}

\begin{document}

\begin{center}
 \uppercase{\bf Congruences for certain lacunary sums of products of binomial coefficients}
\vskip 20pt
{\bf Ren\'{e} Gy}\\
{\tt rene.gy@numericable.com}\\
\end{center}
\vskip 20pt

\vskip 30pt

\centerline{\bf Abstract}
It is shown that for any prime $p$ and any natural numbers $\ell, m,$ and $s$ such that $0<s<p$, the three following congruences \begin{align*}\sum_{i\ge \ell+1}(-1)^{m-i} {m \choose i}{m+s-1+i(p-1) \choose m+s-1+\ell(p-1)} &\equiv 0 \bmod p\\ \sum_{i\ge 0}(-1)^{m-i} {m \choose i}{\ell+ip \choose m+s-1}&\equiv 0 \bmod  p^m\\ \sum_{j,i\ge \ell}(-1)^{j-i}{m \choose j} {j \choose i}{j+s-1+i(p-1) \choose j+s-1+\ell(p-1)}&\equiv 0 \bmod p^{m-\ell} \end{align*} hold true. The corresponding quotients involve Adelberg polynomials which can be computed explicitly, providing closed-form expressions for these sums, valid even if $p$ is not prime, when the congruences do not necessarily hold.

\pagestyle{myheadings} 
\thispagestyle{empty} 
\baselineskip=12.875pt 
\vskip 30pt

\section{Introduction, notations and preliminaries}
There exist many congruences involving binomial coefficients. Apart from their classical arithmetic properties like the famous Kummer, Lucas and Wolstenholme theorems so to name a few, other congruences involving lacunary sums or lacunary sums of products of binomial coefficients have also been known for a long time. These kind of results can be found for instance in the introduction of \cite{Granville97} where advanced arithmetic properties of binomial coefficients are presented. They have been further investigated and generalized in \cite{Sun06} and \cite{Sun07}. We recall two old examples taken from \cite{Granville97} and two more recent examples from \cite{Sun06} and \cite{Sun07}:
\begin{align*}
\sum_{i\ge0}{s+\ell(p-1) \choose h+ i(p-1)} &\equiv {s \choose h}  \pmod p & &\text{(Glaisher, 1899)}\\
\sum_{i\ge0}(-1)^{i p}{s+q(p-1) \choose h+i p} &\equiv 0  \pmod {p^q} & &\text{(Fleck, 1913)}\\
\sum_{i\ge0}(-1)^{i p}{i \choose \ell}{\ell p +s+q(p-1) \choose h+i p} &\equiv 0  \pmod {p^q} & &\text{(Wan, 2005)}\\
\sum_{i,j\ge0}(-1)^{j+i(p-1)}{q\choose j}{h +j(p-1) \choose s+ i(p-1)} &\equiv 0  \pmod {p^q} & &\text{(Sun, Tauraso, 2007)}
\end{align*}
which are valid for prime $p$, non-negative integers $\ell,q$ and integers $s,h$ such that $0< s<p$ and $0\le h<p$. Note that the afore-mentioned congruences are written here differently from how they read in the quoted papers, for an easier comparison in between them and with our own results. Also note that the last one is just a particular case from a vast generalization (\cite{Sun07}, Theorem 1.2).\\

\indent The purpose of the present paper is to establish three new congruences somewhat reminiscent of, but different from the above congruences. Namely, we will show that, for any prime $p$ and any natural numbers $\ell, m$ and $s$ such that $0<s < p$, it holds that 
\begin{align*} 
	\sum_{i\ge \ell+1}(-1)^{m-i} {m \choose i}{m+s-1+i(p-1) \choose m+s-1+\ell(p-1)}&\equiv 0 \pmod p,\\
	\sum_{i\ge 0}(-1)^{m-i} {m \choose i}{\ell+ip \choose m+s-1}&\equiv 0 \pmod { p^m},\\
	\sum_{j,i\ge \ell}(-1)^{j-i}{m \choose j} {j \choose i}{j+s-1+i(p-1) \choose j+s-1+\ell(p-1)}&\equiv 0 \pmod { p^{m-\ell}}
	\end{align*}
	and we will show how to effectively obtain the corresponding quotients.\\
	\ \\
\indent In the following $[[x^n]]f(x)$ denotes the coefficient of $x^n$ in $f(x)$, where $f$ is a formal power series with the argument $x$ and $\partial f(x)$ is the derivative of $f(x)$ with respect to $x$. If $x$ is a real number, we denote $\lfloor x \rfloor$ the largest integer smaller or equal to $x$. We also use the Iverson bracket notation: $\big[\mathfrak{P}\big]=1$ when proposition $\mathfrak{P}$ is true, and  $\big[\mathfrak{P}\big]=0$ otherwise. We recall some basic properties of the binomial coefficients and Stirling numbers, which can be found for instance in \cite{Graham94}. The  binomial coefficients  ${n\choose k}$, are defined by  $\sum_{k}{n\choose k}x^k =(1+x)^n$,  whatever the sign of integer $n$. They obviously vanish when $k<0$. They are easily obtained by  the basic recurrence relation ${{n}\choose {k}}={{n-1}\choose {k}}+{{n-1}\choose{k-1}}$,  they satisfy the Vandermonde convolution: $ \sum_{j=0}^{m}{n\choose j}{k \choose m-j} = {n+k \choose m}$ and when $n  > 0$, we have ${-n\choose k}= (-1)^k{n+k-1\choose n-1}$. The  cycle Stirling numbers (or Stirling numbers of the first kind) ${n\brack k}$, $n\ge0$, may be defined by the horizontal generating function
\begin{align}\label{gf}  \sum_{k}{n\brack {k}}x^k&=\prod_{j=0}^{n-1}(x+j), \end{align}
where an empty product is meant to be $1$. Alternatively, they have the exponential generating function 
\begin{align} \label{egfs1}
 \sum_n {n \brack k} \frac{x^n}{n!}&=\frac {(-1)^k \big(\ln(1-x)\big)^k}{k!}.\end{align}
They obviously vanish when $k<0$ and $k>n$. They are easily obtained by the basic recurrence ${{n}\brack {k}}=(n-1){{n-1}\brack {k}}+{{n-1}\brack {k-1}}$, valid for $n \ge 1$, with ${{0}\brack {k}}=[k=0]$. We let  ${n\brace {k}}$, $n\ge0$, be the partition Stirling numbers (or Stirling numbers of the second kind). They also vanish when $k<0$ and $k>n$. Their basic recurrence is
${{n}\brace {k}}=k{{n-1}\brace {k}}+{{n-1}\brace {k-1}}$ for $n\ge1$, with $ {0\brace {k}}=[k=0]$. 
\noindent They have the following exponential generating function
\begin{align} \label{egfs2}  \sum_n {n \brace k} \frac{x^n}{n!}&=\frac {(e^x-1)^k}{k!}\end{align}
and the following explicit expression
\begin{align}\label{aaa} {n\brace {k}}&= \frac{(-1)^k}{k!}\sum_{j\ge 0}(-1)^j\binom{k}{j}j^n.\end{align}

\noindent We will also need the following knowm lemmas, for which we include a short proof.
\begin {lem}\label{ce}
Let $f(w)$ be a formal power series, and $\alpha$ a natural number. We have $[[w^n]]\frac{f(w)^\alpha}{\alpha}=[[w^{n-1}]]\frac{f(w)^{\alpha-1} \partial f(w)}{n}$.
\end{lem}
\begin{proof} This is clear, since $[[w^n]]f(w)=\frac{[[w^{n-1}]]\partial f(w)}{n}$.
\end{proof}
\begin {lem}\label{st2}
For any natural number $n$, $n>0$, we have \begin{equation} \label{lab} {n \brace {p-1}} \equiv \big[p-1 \text{ divides }n\big]  \pmod p \text{.} \end{equation} 
\end{lem}
\begin{proof} 
We recall the Wilson theorem which states that  $(p-1)! \equiv -1 \bmod p$ for any prime $p$, and two other well-known congruences, valid for any prime $p$:
\begin {align*}
\binom{p-1}{j} &\equiv (-1)^j \big[0\le j\le p-1 \big]\pmod p\text{,}   \\
 \sum_{p-1\ge j\ge 1}j^k  &\equiv  - \big[p-1 \text{ divides }k\big]\pmod p\text{,} 
\end{align*} so that, the claim readily follows from Equation \eqref{aaa}.
\end{proof}

\section {The $p$-congruence for Stirling numbers of the first kind.}
The following known \cite{Peele93} $p$-congruence for the Stirling numbers of the first kind will be essential to our argument in Section 4.
\begin{thm} \label{3.1} Let $p$ be a prime number and $n, k$  non-negative integers such that $0\le k \le n$ and let $r$ (respectively $q$)  be the residue (respectively the quotient) of the Euclidean division of $n$ by $p$ and let $\rho$ be the residue of the Euclidean division of $k-q$ by $p-1$. We have  
\begin{align}\label{3.1}
{n \brack k} \equiv  (-1)^{q -\frac{k-q-j}{p-1}} {r \brack j} \binom{q}{\frac{k-q-j}{p-1}} \pmod p
\end{align}
with $j=\rho+[\rho=0][r=p-1](p-1)$.
\end{thm}
\begin{proof} For the sake of self-containment, we reproduce the proof from \cite{Peele93}. Let $p$ be a prime number. We consider $\prod_{j=0}^{p-1}(x+j)= \sum_{k}{p\brack {k}}x^k $ as an element of the ring of polynomials of $  \mathbb Z  /  p\mathbb Z$. In that ring, we have  $ \sum_{k}{p\brack {k}}x^k= x^p-x $, since the polynomials on both sides have have  same degree $p$, same coefficient for $x^p$ and same roots:  $0,-1,-2, \cdot \cdot \cdot \cdot, -(p-1)$. In particular, for $k$ such that $1<k \le p-1$, we have ${p\brack {k}}\equiv 0 \bmod p$. Let $n$ be a non-negative integer and $r$ (respectively $q$) be  the residue (respectively the quotient) of the Euclidean division of $n$ by $p$, such that $n=q p+r$, with $0 \le r \le p-1$.
We may explicitly write and regroup the factors of $\prod_{j=0}^{n-1}(x+j)$ so that  
$$\prod_{j=0}^{n-1}(x+j)= \prod_{t=0}^{q-1} \left((x+tp)(x+tp+1)\cdot \cdot\cdot (x+tp+(p-1)\right)  \prod_{u=0}^{r-1}(x+q+u),$$
 with the convention that when $q=0$ or $r=0$ the empty products are meant to be equal to  $1$. Then, reducing modulo $ p$, we have 
\begin{align*} \sum_{k}{n\brack {k}}x^k &\equiv \prod_{t=0}^{q-1} \left(x(x+1)\cdot \cdot\cdot (x+(p-1)\right)  \prod_{u=0}^{r-1}(x+q p+u) \pmod p\\
& \equiv \left(x(x+1)\cdot \cdot\cdot (x+(p-1)\right)^{q} \prod_{u=0}^{r-1}(x+u) \pmod p \\
& \equiv \left(x^p-x\right)^{ q} \prod_{u=0}^{r-1}(x+u) \pmod p \\
& \equiv x^{q} \left(x^{p-1}-1\right)^{q} \prod_{u=0}^{r-1}(x+u) \pmod p \\
& \equiv x^{q} \sum_{m=0}^{q} (-1)^{q-m}{q \choose m} x^{m(p-1)} \prod_{u=0}^{r-1}(x+u) \pmod p \end{align*}
\begin{align*}
\sum_{k}{n\brack {k}}x^{k-q} & \equiv \left(\sum_{m=0}^{q} (-1)^{q-m}{q\choose m} x^{m(p-1)}\right)\left(\sum_{\ell=0}^{r}{r\brack {\ell}} x^{\ell}\right) \pmod p \\
& \equiv \sum_{k}\sum_{\underset{m(p-1)+\ell=k-q}{m,l}}(-1)^{q-m}{q \choose m}{r\brack {\ell}} x^{k-q}\pmod p. \\
\end{align*}
\noindent Hence
\begin{align}{n\brack {k}}&\equiv \sum_{\underset{m(p-1)+\ell=k-q}{m,l}}(-1)^{q-m}{q \choose m}{r\brack {\ell}} \pmod p. 
\end{align}
Now if $p$ divides $n$, we have $r=0$, and there may exist only one possible solution in non-negative integers $ \ell,m$  to the equation $ m(p-1)+\ell=k-q $ which may have a non-zero contribution to the sum on the right hand side of the above congruence: this is when $p-1$ divides $k-\frac{n}{p}$ and we have $\ell=0$ and $m=\frac{k-q}{p-1}$. Otherwise, $1\le r <p$ and since $\ell \le r$ and $r<p$, we have $\ell < p$. But we also have $\ell >0$, since ${r\brack 0}=0$, since $r>0$. 
Then,  there exists at most one solution in non-negative integers $ \ell,m$  to the equation $ m(p-1)+\ell=k-q $.   
Indeed, let $\rho$ be the residue of the Euclidean division of $k-q$ by $p-1$. We have $0 \le \rho <p-1$. If $\rho=0$, then the unique solution is $\ell=p-1$ and $m= \frac{k-q}{p-1}-1$. If $0<\rho \le r$, then the unique solution is $\ell = \rho$ and $m= \frac{k-q-\rho}{p-1}$. And finally, if $r<\rho <p-1$, there is no solution. Putting everything together, we obtain the claimed $p$-congruence of the Stirling numbers of the first kind. \end{proof}

\begin{col}\label{3.1.1.c} Let $p$ be a prime number and $i,m$ and $s$  be three natural numbers. We have
\begin{align}\label{3.1.1}
{m+s+m(p-1) \brack m+s +i(p-1)} &\equiv (-1)^{m-i}{m+\lfloor\frac{s}{p}\rfloor \choose i +\lfloor\frac{s}{p}\rfloor}  \pmod p .
\end{align}
\begin{proof} We apply Theorem \ref{3.1}  with  $n= mp+s$ and $k= m+s +i(p-1)$. We have $q= m+\lfloor\frac{s}{p}\rfloor$ and $k-q= s +i(p-1)-\lfloor\frac{s}{p}\rfloor=(i+\lfloor\frac{s}{p}\rfloor)(p-1)+r$. Then, when $s -p\lfloor\frac{s}{p}\rfloor=r <p-1$ we have $r=\rho<p-1$ and then $j=r$ and then Congruence \eqref{3.2} reduces to Congruence \eqref{3.1.1}. Otherwise $s -p\lfloor\frac{s}{p}\rfloor= r = p-1$ and then $\rho=0$ and then $j=p-1$ and then Congruence \eqref{3.2} also reduces to Congruence \eqref{3.1.1}.
\end{proof}

\end {col}
\section {An identity involving Stirling numbers}
In the following theorem, we present an identity involving binomial coefficients and Stirling numbers of both kinds which we believe is new.
\begin{thm} \label{3.2}Let $p$ be a positive integer and $n, k$  non-negative integers. We have
\begin{align}\label{3.3}
(-1)^{p-1}{n-1 \choose p-1}{n-p+1 \brack k}&= \sum_{i}(-1)^i {k-1+i \choose i}{i \brace p-1} {n \brack i+k}.
\end{align}
\end{thm}

\noindent {\bf Remark.} It is interesting to compare Equation \eqref{3.3} to the identity (6.28) in \cite{Graham94}. The latter is also a three parameters identity, but involving only the second kind of Stirling numbers, which is rather easily obtained from their exponential generating function. We reproduce it hereafter: under the condition that $\ell,m,n \ge 0$, we have 
\begin{align} \label{kn} {\ell+m \choose \ell}{n \brace \ell+m}= \sum_k {k\brace \ell}{n-k \brace m}{n \choose k}. \end{align} 

\noindent If we replace $m$ by $-m$ and $n$ by $-n$ in Equation \eqref{kn}, taking into account that  ${-n \choose k}=(-1)^k{n+k-1\choose k}$ and the known duality ${-a \brace -b}$ = ${b \brack a}$, see \cite{Graham94}, we obtain our Equation \eqref{3.3} with $p-1= \ell$. This shows that Equation \eqref{kn} holds under the only condition $\ell \ge 0$ and that the conditions $m,n \ge 0 $ given in \cite{Graham94} are actually not needed.\\
\ \\
\noindent{\it Proof of Theorem \ref{3.2}.} The following proof, using the coefficient extractor method and the generating functions Equation \eqref{egfs1} and Equation \eqref{egfs2} is due to Marko Riedel \cite{Riedel20}.
\noindent Let $S$ be the right hand side of Equation \eqref{3.3}. We have
\begin{align*}
S&=  \sum_{i=p-1}^{n-k}(-1)^i {k-1+i \choose i}{i \brace p-1} {n \brack i+k}\\
&=   \sum_{i=p-1}^{n-k}(-1)^i {k-1+i \choose i}i![[z^i]]\frac{(e^z-1)^{p-1}}{(p-1)!} n![[w^n]]\frac{\left(\log\frac{1}{1-w} \right)^{i+k}}{(i+k)!}\\
&=  \frac{n!}{(p-1)!(k-1)!} [[w^n]]\sum_{i=p-1}^{n-k}(-1)^i [[z^i]](e^z-1)^{p-1} \frac{\left(\log\frac{1}{1-w} \right)^{i+k}}{(i+k)}.
\end{align*}
\noindent Then by Lemma \ref{ce}, we have 
\begin{align*}
S&=  \frac{(n-1)!}{(p-1)!(k-1)!}[[w^{n-1}]]\frac{1}{1-w}\sum_{i=p-1}^{n-k}(-1)^i[[z^i]](e^z-1)^{p-1} \left(\log\frac{1}{1-w} \right)^{i+k-1}\\
&= \frac{(-1)^{k-1}(n-1)!}{(p-1)!(k-1)!} [[w^{n-1}]]\frac{1}{1-w}\sum_{i=p+k-2}^{n-1}(-1)^{i}[[z^i]]z^{k-1}(e^z-1)^{p-1}\left(\log\frac{1}{1-w} \right)^{i}\\
&= \frac{(-1)^{k-1}(n-1)!}{(p-1)!(k-1)!}[[w^{n-1}]] \frac{1}{1-w}\sum_{i\ge p+k-2}\left(-\log\frac{1}{1-w} \right)^{i}[[z^i]]z^{k-1}(e^z-1)^{p-1}
\end{align*}
since $\left(\log\frac{1}{1-w} \right)^{i}= w^i+ \cdot \cdot$. Now, we have  $(e^z-1)^{p-1}=z^{p-1}+  \cdot \cdot$, then the lowest power of $z$ in the power series of $z^{k-1}(e^z-1)^{p-1}$ is $z^{p+k-2}$ and then
\begin{align*}
S&= \frac{(-1)^{k-1}(n-1)!}{(p-1)!(k-1)!}[[w^{n-1}]] \frac{1}{1-w}\sum_{i\ge 0}\left(-\log\frac{1}{1-w} \right)^{i}[[z^i]]z^{k-1}(e^z-1)^{p-1}\\
&= \frac{(-1)^{k-1}(n-1)!}{(p-1)!(k-1)!} [[w^{n-1}]] \frac{1}{1-w}\left(-\log\frac{1}{1-w} \right)^{k-1}(e^{-\log\frac{1}{1-w}}-1)^{p-1}\\
&= \frac{(-1)^{k-1}(n-1)!}{(p-1)!(k-1)!} [[w^{n-1}]] \frac{1}{1-w}\left(-\log\frac{1}{1-w} \right)^{k-1}(-w)^{p-1}\\
&= \frac{(-1)^{p-1}(n-1)!}{(p-1)!(k-1)!} [[w^{n-1}]] \frac{1}{1-w}\left(\log\frac{1}{1-w} \right)^{k-1}w^{p-1}\\
&= \frac{(-1)^{p-1}(n-1)!}{(p-1)!(k-1)!} [[w^{n-p}]] \frac{1}{1-w}\left(\log\frac{1}{1-w} \right)^{k-1}\\
&= \frac{(-1)^{p-1}(n-1)!}{(p-1)!(k-1)!}(n-p+1) [[w^{n-p}]] \frac{1}{1-w}\frac{\left(\log\frac{1}{1-w} \right)^{k-1}}{n-p+1}\\
&= \frac{(-1)^{p-1}(n-1)!}{(p-1)!(k-1)!}(n-p+1)  [[w^{n-p+1}]]\frac{\left(\log\frac{1}{1-w} \right)^{k}}{k} \ \  \ \text{by Lemma  \ref{ce} and then}\\
S&= \frac{(-1)^{p-1}(n-1)!}{(p-1)!(n-p)!}(n-p+1)! [[w^{n-p+1}]]\frac{\left(\log\frac{1}{1-w} \right)^{k}}{k!}\\
&= (-1)^{p-1}{n-1 \choose p-1}{n-p+1 \brack k}.
\end{align*}
\qed

\noindent As corollary, we have a new congruence involving Stirling numbers of the first kind: 
\begin{col}\label{3.4} Let $p$ be a prime number, and let $n,k$ be non-negative integers. The following congruence holds
\begin{align}\label{3.4}
\sum_{\underset{p-1 \vert  i}{i>0}}{k-1+i \choose k-1}{n \brack i+k} &\equiv  \big[p \text{ divides n}\big]{n-p+1 \brack k}  \pmod p .
\end{align}
\begin{proof}
Consider Equation \eqref{3.3} when $p$ is prime. When $i=0$, we have ${i \brace p-1} =0$, and when $i>0$, by Lemma \ref{st2}, we have ${i \brace p-1} \equiv \big[p-1 \text{ divides } i\big] \bmod p$. Moreover  $(-1)^{p-1} \equiv 1 \bmod p$, then, modulo $p$,  the right hand side of Equation \eqref{3.3} is $$\sum_{\underset{p-1 \vert  i}{i>0}}{k-1+i \choose k-1}{n \brack i+k}.$$ Now, for the left hand side of Equation \eqref{3.3}, when $p$ does not divide $n$, by Kummer theorem, whenever $n-p \not \equiv 0 \bmod p$ we have ${n-1 \choose p-1} \equiv 0 \pmod p$, since in this case the addition $(n-p)+(p-1)$ in base $p$ has at least one carry (at the lowest digit). And eventually, when $p$ divides $n$, we have ${n-1 \choose p-1} \equiv {-1 \choose p-1}=(-1)^{p-1} \equiv 1 \pmod p$.
\end{proof}
\end {col}

\section {The first congruence}
\noindent We can now prove our first claim from the introduction. Actually, we have an even stronger theorem:

	\begin{thm}\label{3.1.2} Let $p$ be a prime number, and $m,\ell$ and $s$ three natural numbers such that $m>l$. We have
	\begin{equation*}
	\sum_{i\ge \ell+1}(-1)^{m-i} {m+\lfloor \frac{s}{p}\rfloor \choose i+\lfloor \frac{s}{p}\rfloor}{m+s-1+i(p-1) \choose m+s-1+\ell(p-1)}\end{equation*}
	\begin{align} \label{3.1.2}& \equiv \big[p \text{ divides s}\big](-1)^{m-1-\ell} {m-1+\frac{s}{p} \choose \ell+\frac{s}{p}}\pmod p .
	\end{align}
	\end{thm}
	\begin{proof}  If we substitute $m+s+\ell(p-1)$ for $k$ and $mp+s$ for $n$, Congruence \eqref{3.4} becomes
	\begin{align*}
	&\sum_{\underset{p-1 \vert  i}{i>0}}{ m+s-1+\ell(p-1)+i \choose m+s-1+\ell(p-1)}{mp+s \brack i+ m+s+\ell(p-1)}\\
	 &\equiv  \big[p \text{ divides s}\big]{(m-1)p+s+1 \brack  m+s+\ell(p-1)}  \pmod p ,
	\end{align*}
	which, by the appropriate index change, is
	\begin{align*}
	&\sum_{i\ge\ell +1}{ m+s-1+i(p-1)\choose m+s-1+\ell(p-1)}{mp+s \brack m+s+i(p-1)}\\ & \equiv  \big[p \text{ divides s}\big]{(m-1)p+s+1 \brack  m+s+\ell(p-1)}  \pmod p .
	\end{align*}
	To complete the proof, we make use of Corollary \ref{3.1.1.c} and we have
	\begin{align*}
	&\sum_{i\ge\ell +1}{ m+s-1+i(p-1)\choose m+s-1+\ell(p-1)} (-1)^{m-i}{m+\lfloor\frac{s}{p}\rfloor \choose i +\lfloor\frac{s}{p}\rfloor}\\
	& \equiv  \big[p \text{ divides s}\big](-1)^{m-1-\ell}{m-1+\lfloor\frac{s+1}{p}\rfloor \choose \ell +\lfloor\frac{s+1}{p}\rfloor} \pmod p,
	\end{align*}
	which is the claim, since when $p$ divides $s$, we have $\lfloor\frac{s+1}{p}\rfloor=\frac{s}{p}$.
	\end{proof}
	
	\noindent {\bf Remark.}  Clearly, when $s=0$, we obtain 
	\begin{align} \label{3.1.4}
	\sum_{i\ge \ell+1}(-1)^{i} {m \choose i}{m-1+i(p-1) \choose m-1+\ell(p-1)}+ (-1)^{\ell} {m-1 \choose \ell} & \equiv 0 \pmod p, 
	\end{align}
	and whenever $0<s<p$, we have 
	\begin{equation} \label{3.1.5}
	\sum_{i\ge \ell+1}(-1)^{m-i} {m \choose i}{m+s-1+i(p-1) \choose m+s-1+\ell(p-1)} \equiv 0 \pmod p, 
	\end{equation}
	which is our first claim from the introduction.\\
	\ \\ Finally, we also have the following corollary:
	\begin{col}
	For any prime $p$ and natural numbers $\ell, n$ such that $p$ does not divides $n$, we have 
	\begin{equation} \label{3.1.3}
	\sum_{i\ge \ell+1}(-1)^i {n-r\choose i}{n-1+i(p-1) \choose n-1+\ell(p-1)} \equiv 0\pmod p ,
	\end{equation}
	where $r$ is the non-zero residue of the Euclidean division of $n$ by $p$.
	\end{col}
	\begin{proof} Congruence \eqref{3.1.3} is obtained from Congruence \eqref{3.1.2} in which we let $s$ be
	the residue of the Euclidean division of $n$ by $p$ and we substitute $\lfloor \frac{n}{p}\rfloor p $ for $m$. \end{proof}
	
\section {Adelberg polynomials}
In this section we recall the definition and some properties of the {\it Adelberg polynomials}. Most of the content of this section is taken from  \cite{Adelberg95}. We have the A-Adelberg polynomials and the B-Adelberg polynomials. By definition, the A-Adelberg polynomial $A_{s-1}(x,y,m)$ is \begin{equation}A_{s-1}(x,y,m)\coloneqq \frac{1}{y^m}\sum_{k=0}^m(-1)^{m-k}{m \choose k}{x+ky \choose m+s-1}.\end{equation}  It is the $m$-th {\it divided difference with increment }$y$ of the function $f(x):={x \choose m+ s-1}$. That is 
\begin{equation}\label{a0} A_{s-1}(x,y,m)=\nabla_y^m{x \choose m+ s-1},\end{equation} where $\nabla_yf(x)=\frac{f(x+y)-f(x)}{y}$. The B-Adelberg polynomial $B_{s-1}(y,m)$ is defined as \begin{equation}B_{s-1}(y,m)\coloneqq A_{s-1}(0,y,m)=\frac{1}{y^m}\sum_{k=0}^m(-1)^{m-k}{m \choose k}{ky \choose m+s-1}.\end{equation} $A_{s-1}(x,y,m)$ and $B_{s-1}(y,m)$ are polynomials in $x$, $y$ and, most interestingly, in $m$, of degree $s-1$. We shall re-prove this, after \cite{Adelberg95}, for the sake of self-containment. This proof actually consists in deriving an explicit expression for $B_{s-1}(y,m)$ that will be the key to the proofs of our congruences in the next section. Note also that since the $m$-th difference of a polynomial of degree less than $m$ is zero, we see from these definitions that both $A_{s-1}$ and $B_{s-1}$ vanish when $s<1$. Then, substituting the Vandermonde convolution $ \sum_{j=0}^{m+s-1}{x\choose j}{ky \choose m+s-1-j} = {x+ky \choose m+s-1}$ into the definition of $A_{s-1}(x,y,m)$ we obtain
\begin{equation}\label{a1}A_{s-1}(x,y,m)=\sum_{j=0}^{s-1} {x \choose j}B_{s-1-j}(y,m).\end{equation}

\noindent Now an explicit expression for $B_{u}(y,m)$ is given:
	\begin{thm}\label{Adel1} Let $u$ be a positive integer. We have
	\begin{equation*}
B_u(y,m)=\sum {m\choose t_u}{m-t_u\choose t_{u-1}}\cdot \cdot {m-t_u-t_{u-1}\cdot\cdot-t_2 \choose t_1} \frac{{y-1\choose 1}^{t_1}}{2^{t_1}}\frac{{y-1\choose 2}^{t_2}}{3^{t_2}}\cdot \cdot \frac{{y-1\choose u}^{t_u}}{(u+1)^{t_u}}
	 \end{equation*}
	where the sum is over all the partitions of the integer $u$, that is over all the $u$-uples of non negative integers $(t_1,t_2,...t_u)$ such that $\sum_{i=1}^u it_i=u$. 
	\end{thm}
\noindent This theorem has the following two corollaries:  
\begin{col}\label{Adel2} $B_u(y,m)$ is a polynomial in both $m$ and $y$ of degree $u$  and $A_u(x,y,m)$  is a polynomial of degree $u$ in $m$, in $x$ and in $y$. Moreover, for $u>0$, $m=0$ and $y=1$ are roots of the polynomials  $B_{u}(y,m)$.
\begin{proof} From the explicit expression of Theorem \ref{Adel1}, $B_u(y,m)$ is clearly a polynomial of both $y$ and $m$, and so is $A_u(x,y,m)$ after Equation \eqref{a1}. The degree of $B_u$  is clearly $u$ since  $t_1+t_2+\cdot\cdot+t_u$ is less than or equal to $u$ and the equality is reached for the partition made with $1$s only. Then, from Equation \eqref{a1}, the degree of $A_u$ is $u$ also. Now when $u>0$, $y=1$, is clearly a root of  $B_u$ since no partition of the integer $u>0$ has all the $t_j=0$. Also, for each partition of $u$ positive, there is a maximal $j>0$ such that $t_j \neq 0$, so that ${m \choose t_j}$ factors out of the corresponding summand, and then $m$ factors out of the whole sum. 
\end{proof}
\end {col}
\begin{col}\label{Adel3} Let $p$ be prime, let $m,x \in  \mathbb{Z}$ and let $s$ be an integer such that $0<s<p$. We have $B_{s-1}(p,m) \in \mathbb{Z}$ and $A_{s-1}(x,p,m) \in \mathbb{Z}$ .
\begin{proof} We only need a proof for the B-polynomial, since the proof for the A-polynomial will then follow from Equation \eqref{a1}. In the explicit expression of $B_{s-1}(p,m)$ from Theorem \ref{Adel1}, each factor $\frac{{p-1\choose j}}{j+1}=\frac{{p\choose j+1}}{p}$ is clearly and integer when $s-1 <p-1$, since $j\le s-1$.
\end{proof}
\end {col}

\noindent{\it Proof of Theorem \ref{Adel1}.} This is the proof from \cite{Adelberg95}. Introducing the Vandermonde convolution again, the first divided difference with increment $y$ of ${x \choose r}$ is 
$$ \nabla_y{x \choose r}=\frac{{x+y \choose r}-{x \choose r}}{y}=\frac{\sum_{j=0}^r{x \choose j}{y \choose r-j}-{x \choose r}}{y}=\frac{\sum_{j=0}^{r-1}{x \choose j}{y \choose r-j}}{y}.$$ Then, iterating $m$ times the divided difference operator, we have
\begin{align*} \nabla_y^m{x \choose r}&=\frac{1}{y^m}\sum_{j_1=0}^{r-1}\sum_{j_2=0}^{j_1-1}\cdot \cdot \sum_{j_m=0}^{j_{m-1}-1}{x \choose j_m}{y \choose j_m-j_{m-1}}\cdot \cdot{y \choose j_2-j_1}{y \choose r-j_1}\\
&=\frac{1}{y^m}\sum_{j=0}^{r-m}{x \choose j}\sum_{\underset{k_1+k_2\cdot\cdot+k_m=r-j}{k_i>0}}{y \choose k_m}\cdot \cdot{y \choose k_2}{y \choose k_1},
\end{align*}
where the inner sum in the last line is over all the $m$-uples of positive integers whose sum is $r-m$. Then by Equation \eqref{a0}, we have
\begin{align*}
A_{u}(x,y,m)&=\sum_{j=0}^{u}{x \choose j}\frac{1}{y^m}\sum_{\underset{k_1+k_2\cdot\cdot+k_m=m+u-j}{k_i>0}}{y \choose k_m}\cdot \cdot{y \choose k_2}{y \choose k_1}\\
&=\sum_{j=0}^{u}{x \choose j}\sum_{\underset{k_1+k_2\cdot\cdot+k_m=m+u-j}{k_i>0}}\frac{{y-1 \choose k_m-1}\cdot \cdot{y-1 \choose k_2-1}{y-1 \choose k_1-1}}{k_m\cdot\cdot k_2k_1}\\
&=\sum_{j=0}^{u}{x \choose j}\sum_{\underset{k_1+k_2\cdot\cdot+k_m=u-j}{k_i\ge0}}\frac{{y-1 \choose k_m}\cdot \cdot{y-1 \choose k_2}{y-1 \choose k_1}}{(k_m+1)\cdot\cdot (k_2+1)(k_1+1)}.
\end{align*}
Then, by comparison with Equation \eqref{a1}, we have 
\begin{equation} 
\label{b1}
B_{u}(y,m)=\sum_{\underset{k_1+k_2\cdot\cdot+k_m=u}{k_i\ge0}}\frac{{y-1 \choose k_m}\cdot \cdot{y-1 \choose k_2}{y-1 \choose k_1}}{(k_m+1)\cdot\cdot (k_2+1)(k_1+1)}
\end{equation}
where the sum is over all the {\it weak compositions} of $u$ with $m$ non negative integers. But it is well known that there are ${m \choose t_0,t_1,\cdot\cdot,t_m}=\frac{m!}{t_0!t_1!\cdot\cdot t_m!}$ weak compositions of $u$ which produce the same unique {\it partition} of $u$ in $m-t_0$ summands, such that $u=0\cdot t_0+1\cdot t_1+\cdot\cdot\cdot+m\cdot t_m$ and $m=t_0+t_1+\cdot\cdot+t_m$ . Then
\begin{equation} \label{adelalt}
B_{u}(y,m)=\sum_{\underset{t_1+t_2\cdot\cdot+t_{s-1}=m-t_0}{t_1+2t_2\cdot\cdot+ut_{u}=u}}\frac{m!}{t_0!t_1!\cdot\cdot t_{u}!}\left(\frac{{y-1 \choose 0}}{1}\right)^{t_0}\cdot \cdot\left(\frac{{y-1 \choose u}}{u+1}\right)^{t_{u}}.
\end{equation}
Now, the multinomial coefficient $\frac{m!}{t_0!t_1!\cdot\cdot t_m!}$ can be expanded as a finite product of binomial coefficients so that
	\begin{equation*}
B_u(y,m)=\sum {m\choose t_u}{m-t_u\choose t_{u-1}}\cdot \cdot {m-t_u-t_{u-1}\cdot\cdot-t_2 \choose t_1} \frac{{y-1\choose 1}^{t_1}}{2^{t_1}}\frac{{y-1\choose 2}^{t_2}}{3^{t_2}}\cdot \cdot \frac{{y-1\choose u}^{t_u}}{(u+1)^{t_u}}
	 \end{equation*}
	where the sum is over all the partitions of the integer $u$, that is over all the $u$-uples of non negative integers $(t_1,t_2,...t_u)$ such that $\sum_{i=1}^u it_i=u$. 
\qed 
\\ \ \\
\indent From Equation \eqref{b1}, we see that $B_u(y,m)$ is the coefficient of $z^u$ in the power series expansion of $\left(\sum_j \frac{{y-1 \choose j}}{j+1}z^j\right)^m=\left(\frac{(1+z)^y-1}{yz}\right)^m$, so that we have the following generating functions for the Adelberg polynomials \cite{Adelberg95}:
\begin{equation}
\sum_{u \ge 0} B_u(y,m)z^u=\left(\frac{(1+z)^y-1}{yz}\right)^m
\end{equation}
\begin{equation} \label{gf2}
\sum_{u \ge 0} A_u(x,y,m)z^u=\left(1+z\right)^x\left(\frac{(1+z)^y-1}{yz}\right)^m
\end{equation}
where the latter is obtained from the former after accounting for Equation \eqref{a1}. Adelberg gives many {\it symetries} (or identities) for his polynomials. We will need the symetry (Sxvi) from \cite{Adelberg95}. It reads
\begin{equation} \label{s3}
\sum_{j \ge 0}{m \choose j}y^jA_u(x+j,y,j)=(y+1)^mA_u(x,y+1,j,m)
\end{equation}
and is obtained with the generating function \eqref{gf2}, as explained in \cite{Adelberg95}: we just need to verify that 
$ \sum_{j \ge 0}{m \choose j}y^j(1+z)^{x+j}\left(\frac{(1+z)^y-1}{yz}\right)^j=(y+1)^m(1+z)^x \left(\frac{(1+z)^{y+1}-1}{(y+1)z}\right)^m$, which is elementary.
 
\section {The second and third congruences}
We are now ready for our second and third congruences.
\begin{thm} Let $p$ be prime and $\ell, m, s$ non-negative integers such that $0<s<p$. We have 
\begin{equation} \label{second}\sum_{i\ge 0}(-1)^{m-i} {m \choose i}{\ell+ip \choose m+s-1}\equiv 0 \pmod { p^m} \end{equation}
and the quotient is an integer-valued polynomial function of $m$ of degree $s-1$.
\begin{proof} Everything has already been done in the previous section: by definition, the quotient is the Adelberg polynomial $A_{s-1}(\ell,p,m)$ which is integer-valued by Corollary \ref{Adel3}.
\end{proof}
\end{thm}
\indent Our third congruence requires slightly more work, as we will need the following generalization of Corollary \ref{Adel3}:
	\begin{thm} \label{gen} Let $p$ be prime, $\ell,m$ non-negative integers and $s$ an integer such that $0<s<p$. We have  $p^\ell B_{\ell(p-1)+s-1}(p,m) \in \mathbb{Z}$ and $p^\ell A_{\ell(p-1)+s-1}(s-1,p,m) \in \mathbb{Z}$. In other words, for any non negative integer $u$, $r$ being the residue of the Euclidean division of $u$ by $p-1$, we have  $p^{\lfloor \frac{u}{p-1}\rfloor}B_u(p,m) \in \mathbb{Z}$ and $p^{\lfloor \frac{u}{p-1}\rfloor}A_{u}(r,p,m) \in \mathbb{Z}$.
	\begin{proof} We may rewrite slightly differently Equation \eqref{adelalt} for $B_u(y,m)$ in the case where $y=p$ a prime number :
	\begin{align*} 
	B_{u}(p,m)&=\sum_{t_1+2t_2\cdot\cdot+ut_{u}=u}\frac{m!}{t_0!t_1!\cdot\cdot t_{u}!}\left(\frac{{p \choose 2}}{p}\right)^{t_1}\cdot \cdot\left(\frac{{p \choose u+1}}{p}\right)^{t_{u}}.
	\end{align*}
	The partitions  of $u$ for which there exist $j\ge p$ such that $t_j \neq 0$ do not contribute to the sum because when $j\ge p$, we have ${p \choose j+1}=0$ and then we have 
	\begin{align*} 
	B_{u}(p,m)&=\sum_{t_1+2t_2\cdot\cdot+(p-1)t_{p-1}=u}\frac{m!}{t_0!t_1!\cdot\cdot t_{p-1}!}\left(\frac{{p \choose 2}}{p}\right)^{t_1}\cdot \cdot\left(\frac{{p \choose p}}{p}\right)^{t_{p-1}}\\
	&=\sum_{t_1+2t_2\cdot\cdot+(p-1)t_{p-1}=u}\frac{m!}{t_0!t_1!\cdot\cdot t_{p-1}!}\left(\frac{{p \choose 2}}{p}\right)^{t_1}\cdot \cdot\left(\frac{{p \choose p-1}}{p}\right)^{t_{p-1}}\frac{1}{p^{t_{p-1}}}.
	\end{align*}
	Now, we let $u=\ell(p-1)+s-1$ and we obtain 
	
	\begin{align*} 
	p^\ell B_{\ell(p-1)+s-1}(p,m)&=\sum_{}\frac{m!}{t_0!t_1!\cdot\cdot t_{p-1}!}\left(\frac{{p \choose 2}}{p}\right)^{t_1}\cdot \cdot\left(\frac{{p \choose p-1}}{p}\right)^{t_{p-1}}p^{\ell - t_{p-1}}
	\end{align*}
	where the sum is over all the $(p-1)$-uples of non-negative integers ($t_1,\cdot \cdot t_{p-1}$) such that $t_1+2t_2\cdot\cdot+(p-1)t_{p-1}=\ell(p-1)+s-1$. We see that $p^\ell B_{\ell(p-1)+s-1}(p,m) \in \mathbb{Z}$ because if $ t_{p-1} \ge \ell+1$ then we would have $\ell(p-1)+s-1 \ge (p-1)t_{p-1}\ge \ell(p-1)+p-1$, which is not possible since it is supposed that $s<p$. Now it follows that $p^\ell A_{\ell(p-1)+s-1}(s-1,p,m)$ is integer because
	\begin{align*}
	p^\ell A_{\ell(p-1)+s-1}(s-1,p,m)&=\sum_{j=0}^{\ell(p-1)+s-1}{s-1\choose j} p^\ell B_{\ell(p-1)+s-1-j}(p,m)\\
	&=\sum_{j=0}^{s-1}{s-1 \choose j} p^\ell B_{\ell(p-1)+s-1-j}(p,m).
	\end{align*}
	\end{proof}
	\end{thm}

\begin{thm} \label{cong3} Let $p$ be prime and $\ell, m, s$ integers such that $0<s<p$. We have 
\begin{align*}\sum_{j,i\ge \ell}(-1)^{j-i}{m \choose j} {j \choose i}{j+s-1+i(p-1) \choose j+s-1+\ell(p-1)}&\equiv 0 \pmod { p^{m-\ell}}. \end{align*}
The quotient is an integer-valued polynomial function of $m$ of degree $s-1+\ell(p-1)$.
\begin{proof} We start from Equation \eqref{s3} where we let $y=p-1$ and $u=\ell(p-1)+s-1$. We replace the Adelberg polynomial on the right hand side by its original definition, rearrange the sums and then we obtain
\begin{align*}
\sum_{j,i\ge \ell}(-1)^{j-i}{m \choose j} {j \choose i}{j+s-1+i(p-1) \choose j+s-1+\ell(p-1)}&=p^mA_{\ell(p-1)+s-1}(s-1,p,m)\\
&=p^{m-\ell}p^\ell A_{\ell(p-1)+s-1}(s-1,p,m).
\end{align*}
The claim then follows from Theorem \ref{gen}.
\end{proof}
\end{thm}
\noindent {\bf Remark.} When $0 \le n <\ell$, it is clear that $A_{\ell(p-1)+s-1}(s-1,p,n)=0$ and  we have $p^\ell A_{\ell(p-1)+s-1}(s-1,p,\ell)=1$. Let 
$$s_{p,s,\ell}(m) \coloneqq \sum_{i\ge \ell}(-1)^{m-i} {m \choose i}{m+s-1+i(p-1) \choose m+ s-1+\ell(p-1)}.$$ 
We have 
$$p^mA_{\ell(p-1)+s-1}(s-1,p,m)=\sum_j {m \choose j}s_{p,s,\ell}(j)$$ 
so that $p^mA_{\ell(p-1)+s-1}(s-1,p,m)$ is a {\it binomial transform} of $s_{p,s,\ell}(m)$ where we consider $ s_{p,s,\ell}(m)$ as an integer sequence with index $m$. Equivalently, by binomial inversion, we have 
        \begin{align*} 
	s_{p,s,\ell}(m)&= \sum_{j}(-1)^{m-j}{m \choose j} p^jA_{\ell(p-1)+s-1}(s-1,p,j)\\
	&= \sum_{j\ge \ell}(-1)^{m-j}{m \choose j} p^{j-\ell}p^\ell A_{\ell(p-1)+s-1}(s-1,p,j)\\
	&= (-1)^{m-\ell}{m\choose \ell}+\sum_{j\ge 1}p^j (-1)^{m-j-\ell}{m \choose j +\ell} p^\ell A_{\ell(p-1)+s-1}(s-1,p,j+\ell)
	\end{align*}
	But the right hand side of Congruence \eqref{3.1.5} is $s_{p,s,\ell}(m)-(-1)^{m-\ell} {m \choose \ell}$ and then our first congruence also follows from the above derivation.
\section {Examples and final remarks}
With appropriate changes of variables and index, the Adelberg polynomials are actually degenerate Stirling or Bernoulli polynomials of arbitrary order \cite{Adelberg95}. But Adelberg writes that his polynomials have {\it a more combinatorial flavor} and {\it show more of the landscape}. His approach also provides an effective way to compute our lacunary sums of products of binomial coefficient, via Theorem \ref{Adel1} and Equation \eqref{a1}. The result of such computations for the first few Adelberg polynomials is displayed in the following tables:    
{\setlength{\tabcolsep}{3pt}
{\center  {\small
	\begin{tabular}{|c|l|}
	\hline
	$u$  & $B_u(y,m)$ \\
	\hline
	& \\
	$0$ & $1$\\
	& \\
	$1$ & $\frac{1}{2}m (-1+y)$\\
	& \\
	$2$ & $\frac{1}{24} m (-1+y) (-5-3 m+y+3 m y)$\\
	& \\
	$3$ & $\frac{1}{48}m (-1+y) (-2-m+m y) (-3-m+y+m y)$\\
	& \\
	$4$& $\frac{1}{5760}m(-1 +y)( -502-485 m-150 m^2-15 m^3+218 y+655 m y+330 m^2y+ 45 m^3y$\\
	& \text{\ \ \ \ \ \ \ \ \ \ \ \ \ \ \ \ \ \ \ \ \ }$-2 y^2-175 m y^2-210 m^2 y^2-45 m^3 y^2-2 y^3+5 m y^3+30 m^2 y^3$\\ 
	& \text{\ \ \ \ \ \ \ \ \ \ \ \ \ \ \ \ \ \ \ \ \ }$+15 m^3 y^3)$\\
	& \\
	\hline 
	\end{tabular} }
\endcenter}
{\center {\bf Table 1:}   The first five B-Adelberg polynomials.\endcenter}
{\setlength{\tabcolsep}{3pt}
{\center  {\small
	\begin{tabular}{|c|l|}
	\hline
	$u$  & $A_u(x,y,m)$ \\
	\hline
	& \\
	$0$ & $1$\\
	& \\
	$1$ & $\frac{1}{2}(-m+2 x+m y)$\\
	& \\
	$2$ & $\frac{1}{24}(5 m+3 m^2-12 x-12 m x+12 x^2-6 m y-6 m^2 y+12 m x y+m y^2+3 m^2 y^2)$\\
	& \\
	$3$ & $\frac{(-2-m+2 x+m y)}{48} (3 m+m^2-8 x-4 m x+4 x^2-4 m y-2 m^2 y+4 m x y+m y^2+m^2 y^2)$\\
	& \\
	\hline 
	\end{tabular} }
\endcenter}
{\center {\bf Table 2:}   The first four A-Adelberg polynomials.\endcenter}
\ \\
\indent We can also write many impressive-looking binomial identities like for instance 
\begin{align*} \label{example}&\sum_{i\ge 0}(-1)^{m-i} {m \choose i}{\ell+in \choose m+2}=\frac{n^m}{24}\big(5 m+3 m^2-12 \ell-12 m \ell+12 \ell^2-6 m n-6 m^2 n\\
& \text{\ \ \ \ \ \ \ \ \ \ \ \ \ \ \ \ \ \ \ \ \ \ \ \ \ \ \ \ \ \ \ \ \ \ \ \ \ \ \ \ \ \ \ \ }+12 m \ell n+m n^2+3 m^2 n^2\big)\end{align*}
or 
\begin{align*}
\sum_{j,i\ge0}(-1)^{j-i}{m \choose j}{j \choose i}{j+\small{\text{5}}+ \small{\text{6}}i \choose j+\small{\text{5}}}&= \small{\text{7}}^{m}\frac{(m\small{\text{+1}})(\small{\text{81}}m^4\small{\text{+684}}m^3\small{\text{+1401}}m^2\small{\text{+434}}m\small{\text{+40}})}{\small{\text{40}}}.
\end{align*}
The polynomial part on the right hand side of these identities is integer-valued when the number which is raised to the power $m$ is prime.\\
\ \\
\indent We conclude the paper by a comparison of our {\it Adelberg}-congruences with the Fleck-like congruences reported in the introduction. They are similar in the fact that for both kinds, the most inner sum is indeed lacunary: one of the two indices of the binomial coefficient is from an arithmetic progression with a ratio which larger than $1$, but they differ by which index: in the Fleck-like congruences, the lower index of the binomial coefficient is lacunary, whereas the upper index is lacunary in our work. To illustrate further this aspect we rewrite Congruence \eqref{second} together with Wan congruence, but in a different form, highlighting the similarity and the difference. For $p$ prime, non-negative integers $m, \ell$ and $r$, under the condition  $m < (p - 1) (\ell + 1)$, we have 
\begin{align*}
\sum_{k \equiv r \bmod p} (-1)^{\frac{k-r}{p}} {\ell(p-1)\choose \frac{k-r}{p}}{k \choose m}& \equiv 0 \ \ \ \ \bmod {p^{\ell(p-1)}},
\end{align*}
whereas, under the condition $r<p$, we have
\begin{align*} \sum_{k \equiv r \bmod p} (-1)^{k-r} {\frac{k-r}{p}\choose \ell}{m \choose k}& \equiv 0 \ \ \ \ \bmod {p^{\lfloor \frac{m-p\ell  -1}{p-1}\rfloor}}.\end{align*}
\ \\
\ \\
{\bf Acknowledgement}: The author is indebted to Marko Riedel for the proof of Theorem \ref{3.2}.

\end{document}